\documentclass[12pt,a4paper,reqno]{amsart}

\usepackage[T1]{fontenc}
\usepackage[utf8]{inputenc}
\usepackage{amsthm, textcomp, amsmath, amsfonts,  amssymb }
\usepackage{graphicx}
\usepackage{enumerate}
\usepackage{url}
\usepackage{comment}
\usepackage[colorlinks]{hyperref}
\usepackage[a4paper,margin=2.2cm]{geometry}
\allowdisplaybreaks[4]
\usepackage{ulem}
\usepackage{xcolor}

\newtheorem{theorem}{Theorem}
\newtheorem*{theorem*}{Theorem}

\newtheorem{definition}[theorem]{Definition}

\newtheorem{lemma}[theorem]{Lemma}

\newtheorem{proposition}[theorem]{Proposition}
\theoremstyle{definition}
\newtheorem{remark}[theorem]{Remark}

\usepackage{orcidlink}
\title{The pearl ensemble and the logarithmic energy on the Sphere}

\author[B. Gariboldi]{B. Gariboldi\textsuperscript{1} \orcidlink{0000-0001-8714-4135}}
\address{\textsuperscript{1}Facoltà di Ingegneria e Informatica, Dipartimento di Scienze e tecnologie dell’Informazione, Universit{\`a} Unipegaso, Italy}
\email{biancamaria.gariboldi@unipegaso.it}
\author[G. Gigante]{G. Gigante\textsuperscript{2} \orcidlink{0000-0002-1642-679X}}
\address{\textsuperscript{2}Dipartimento di Ingegneria Gestionale, dell'Informazione e
della Produzione, Universit{\`a} degli Studi di Bergamo, Viale G. Marconi 5, 24044,
Dalmine (BG), Italy}
\email{giacomo.gigante@unibg.it}

\thanks{}

\allowdisplaybreaks
\definecolor{zzttqq}{rgb}{0.6,0.2,0.}
\begin{document}

\begin{abstract}
We define a family of random sets of points on the sphere $\mathbb S^2$, the Pearl ensemble, depending on several parameters. The expected value of the logarithmic energy can be computed exactly up to terms of order $O(\sqrt{N}\log N)$, with $N$ the number of points. For properly chosen values of the parameters, the coefficient of the linear term in the expansion of the logarithmic energy can be taken as close as desired to the value $(1-\log 3)/2\sim -0.0493061\ldots$. Among the explicit constructions currently known to us for which the logarithmic energy expansion has been rigorously determined up to the linear term, the Pearl ensemble yields the smallest coefficient.
\end{abstract}
\maketitle

\section{Introduction}

Our starting point is Smale's 7th problem. Its original formulation is as follows. 

\bigskip

\begin{quotation}
\textbf{Problem 7: Distribution of Points on the 2-Sphere} (\cite{Smale}).
Let 
\[V_N(x)=\sum_{1\leq i<j\leq N} \log\frac{1}{\|x_i-x_j\|},\]%\footnote{In [Smale] the logarithmic energy $V_N$ is defined as $\sum_{1\leq i<j\leq N} \log\frac{1}{\|x_i-x_j\|}$. Here we use this different convention, which is used in the more recent literature on the topic.}
where $x=(x_1,\ldots,x_N)$, the $x_i$ are distinct points on the 2-sphere $\mathbb S^2\subset \mathbb R^3$, and $\|x_i-x_j\|$ is the distance in $\mathbb R^3$. Denote $\min_x V_N(x)$ by $V_N$. 

\textit{Find $(x_1,\ldots,x_N)$ such that
\begin{equation}\label{smale}
V_N(x)-V_N\leq c\log N, \quad c\text{ a universal constant}.\end{equation}}
To ``find'' means to give an algorithm which on input $N$ outputs distinct $x_1,\ldots,x_N$ on the 2-sphere satisfying \eqref{smale}. 
[$\ldots$]
The function $V_N$ as a function of $N$ satisfies 
\[V_N=-\frac14\log\left(\frac4e \right)N^2-\frac14 N\log N+O(N).\]
[$\ldots$]
\end{quotation}
\bigskip

In this paper, following the convention used in the more recent literature, we consider the logarithmic energy
\[E_{N}(x)=2V_N(x)=\sum_{i\neq j} \log\frac{1}{\|x_i-x_j\|}.\]
We denote $\min_{x} E_N(x)$ by $\mathcal E_N$

In \cite{BS} the authors prove that there exists a constant $C\neq 0$ such that if $N\to+\infty$
\[\mathcal E_N=\left(\frac12-\log 2 \right)N^2-\frac12 N\log N+ C N+o(N),\]
where 
\[C\leq C_{BHS}:=2\log 2+\frac12\log \frac23 +3\log \frac{\sqrt{\pi}}{\Gamma(1/3)}=-0.0556053\ldots,\]
and it was conjectured originally by Brauchart, Hardin and Saff in \cite{BHS} that $C=C_{BHS}$.
Recently, J. Marzo in \cite{JM} proved that 
\[C\geq \log 2-\frac34+\frac1{162}\left(\sqrt[4]{3}\sqrt{2\pi}\left(2+3\tanh^{-1}\frac12 \right)-12 \right)^2=-0.0568456\ldots.\]

We are therefore far from the solution to Smale's 7th problem since it is not even clear what the constant $C$ exactly is. Nevertheless, several attempts at finding more or less explicit point configurations with low logarithmic energy have been made by several authors. We recall here three types of configurations coming from different point processes. The spherical ensemble (see \cite{AZ}) gives an average logarithmic energy
\[\mathbb{E}(E_N)=\left(\frac12-\log2 \right)N^2-\frac12N\log N+\left(\log 2-\frac\gamma2 \right)N-\frac14+O\left(\frac1N \right),\]
where $\gamma$ is the Euler constant. Notice that
\[\log 2-\frac\gamma2=0.404539\ldots.\]
On the other hand, in \cite{ABS} it is shown that points obtained as projection of zeros of certain random polynomials give an average logarithmic energy
\[\mathbb{E}(E_N)=\left(\frac12-\log2 \right)N^2-\frac12N\log N-\left(\frac12-\log2 \right)N+o(N),\]
where
\[-\left(\frac12-\log2 \right)=0.193147\ldots.\]
Finally, the best result so far comes from the Diamond ensemble in the papers of Beltr\'an, Etayo and L\'opez-G\'omez (see \cite{BE}, \cite{Pedro}, \cite{PedroTesi}) which gives an average logarithmic energy
\begin{equation}\label{BE}\mathbb{E}(E_N)=\left(\frac12-\log2 \right)N^2-\frac12N\log N+c_{\diamond}N+o(N),
\end{equation}
with
\[c_{\diamond}=-0.049222\ldots.\]
Let us take a closer look at the construction in \cite{BE}. The point distribution is composed by the North and South poles and a collection of points placed on different parallels. If one wants to keep the points in such a collection well separated, one can proceed as follows: divide the sphere into $2n+1$ parallels separated from each other by an angle $\Delta\varphi=\pi/(2n+2)$ and place on the $j$-th parallel $v_j$ points spaced by a multiple of the Euclidean distance between two consecutive parallels,  $2\sin(\Delta\varphi/2)$. Since the length of the $j$-th parallel is $2\pi\sin(j\Delta\varphi)$, this would give
\[v_j=\frac{K_0\pi\sin(j\Delta\varphi)}{\sin(\Delta\varphi/2)}.\]
The randomness in this construction comes from a random rotation of each parallel. 
The problem of this construction is of course that $v_j$ in general is not an integer. A heuristic argument that ignores this issue shows that the choice $K_0=3/\pi$ would give an average logarithmic energy
\begin{equation}\label{best}
\mathbb{E}(E_N)=\left(\frac12-\log2 \right)N^2-\frac12N\log N+\frac{1-\log 3}{2}N+o(N),
\end{equation}
where
\[\frac{1-\log 3}{2}=-0.049306\ldots.\] 

\begin{figure}
\label{fig}
\includegraphics[width=1\textwidth]{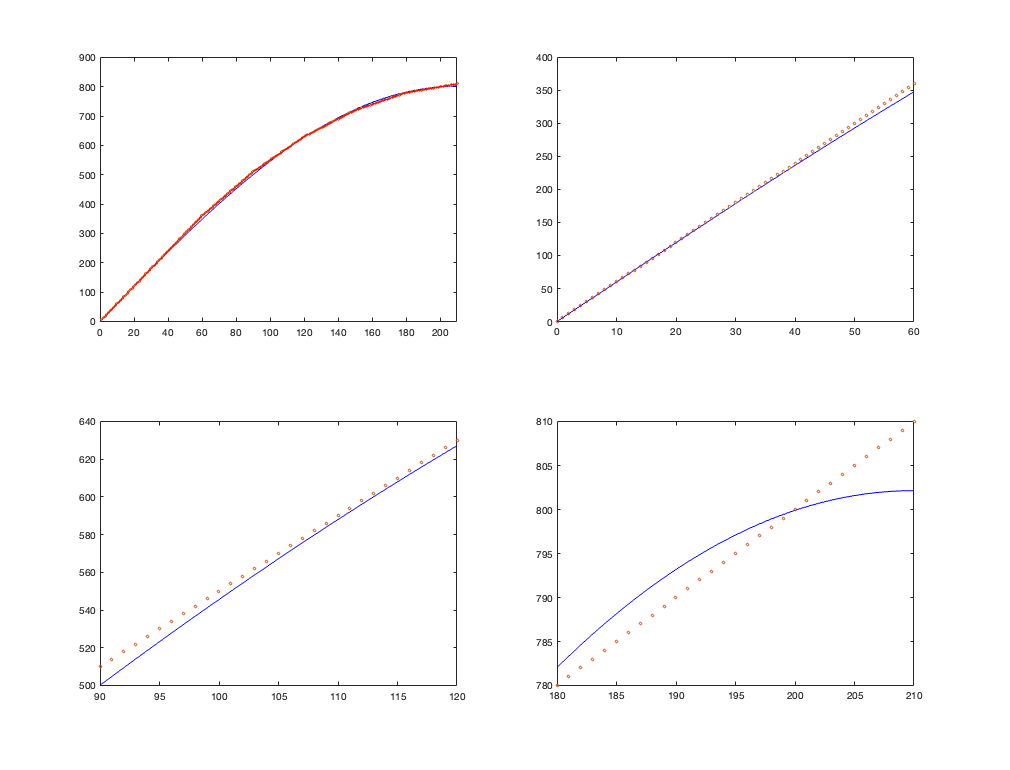}
\caption{The functions $j\to v_j$ (blue) and $j\to r(j)$ (red), with $n=209$.}
\end{figure}

The authors in \cite{BE} replace the function $j\to v_j$ with piecewise linear approximations $j\to r(j)$ taking integer values and such that
\[r(j)=r(2n+2-j)=\begin{cases} \alpha_1+\beta_1 j & \text{if } 0=t_0\leq j\leq t_1\\ \vdots & \vdots\\  \alpha_k+\beta_k j & \text{if } t_{k-1}\leq j\leq t_k=n+1 \end{cases}\]
with $0=t_0<t_1<\ldots<t_{k-1}<t_k=n+1$ and all the $t_\ell, \alpha_\ell, \beta_\ell$ are assumed to be integer numbers. This function has been chosen in such a way that when $n=7m-1$ ($m$ positive integer) and $N=239m^2+2$, then the average logarithmic energy is precisely \eqref{BE}. More precisely, the authors set $k=6$, and for $1\le \ell\le 6$, $t_\ell=(1+\ell)m$, $\alpha_\ell=(\ell^2+\ell-2)m/2$ and $\beta_\ell=7-\ell$. See Figure \ref{fig} for the case $m=30$. Later, in \cite{Pedro} (see also \cite{PedroTesi}) a more clever choice of the parameters $t_\ell, \alpha_\ell, \beta_\ell$ gives \eqref{BE} for every number of points $N$. %We believe they named this point process Diamond ensemble due to the piecewise linear approximation $r(j)$. 
Despite the roughness of the piecewise linear approximation $r(j)$, the constant $c_\diamond$ is a very good approximation of the heuristic constant $(1-\log 3)/2$. 

The authors of \cite{BE} themselves notice that their construction is reminiscent of the one introduced by  Rakhmanov, Saff and Zhou in \cite{RSZ}. In \cite{RSZ} the authors  obtain an equal area partition of the sphere consisting of equal sectors in the two spherical caps centered in the North and in the South poles and rectangular cells obtained first dividing the sphere with essentially equispaced parallels and then dividing each collar between two consecutive parallels with meridians. Indeed, the numerical simulations in \cite[Figure 22]{HMS} on the logarithmic energy obtained considering the centers of each cell give an empirical value of
\[\lim_{N\to+\infty}\dfrac{\mathbb{E}(E_N)-\left(\frac12-\log2 \right)N^2+\frac12 N\log N}N\]
close to $c_\diamond$.

The partition strategy in \cite{RSZ} is the following: for any partition $\{m_1,\ldots,m_{2n+1}\}$ of the integer $N$ the authors first divide the sphere into $2n+1$ collars and then divide the $j$-th collar into $m_j$ equal rectangular cells. Assume the $j$-th collar is bounded by latitudes $\psi_{j-1}$ and $\psi_{j}$, with $\psi_0=0$ and $\psi_{2n+1}=\pi$. Since the area of each cell has to be $4\pi/N$, in order to achieve small diameters the desired length of the sides of each cell is $2\sqrt{\pi/N}$. Hence the angle $\psi_j-\psi_{j-1}$ of each collar should be roughly $2\sqrt{\pi/N}$ and therefore $2n+1\sim \sqrt{\pi N}/2$. On the other hand, since the area of the $j$-th collar is $2\pi(\cos\psi_{j-1}-\cos\psi_j)$, each collar should be divided into 
\[m_j \sim y_j:=\dfrac{2\pi(\cos\psi_{j-1}-\cos\psi_j)}{4\pi/N}=\frac N2(\cos\psi_{j-1}-\cos\psi_j).\]  
As for the Diamond ensemble, also in this case the authors face the problem that $y_j$ in general is not an integer. Rather than using a piecewise linear approximation, in this case $y_j$ is replaced by one of the two closest integers by means of the following lemma.

\begin{lemma}\cite[Lemma 2.5]{RSZ} \label{lemmanondopato}
For every $n$ let $\{y_i\}_{i=1}^{2n+1}$ be a sequence of real numbers such that
\[y_j=y_{2n+2-j}, \qquad j=1, \ldots, n\] and such that \[\sum_{i=1}^{2n+1} y_i=T\in\mathbb N.\] 
Then there exists a symmetric sequence of integers $\{m_i\}_{i=1}^{2n+1}$ such that
\begin{itemize}
\item[$i)$] $\sum\limits_{i=1}^{2n+1} m_i=\sum  \limits_{i=1}^{2n+1} y_i=T$;
\item[$ii)$] $|y_i-m_i|\leq 1, \ i=1, \ldots, 2n+1;$
\item[$iii)$] $\left|\sum\limits_{i=1}^j (y_i-m_i) \right|\leq \dfrac12, \ j=1, \ldots, 2n+1$.
\end{itemize}
\end{lemma}

%seems to deteriorate when the collars approach the poles. This is due 

%Our modest contribution is to search for a sharper approximation of the function $j\to y_j$ in order to obtain an average logarithmic energy as close as possible to \eqref{best}. The idea is very simple and consists in using a slight improvement of the following lemma by Saff, Rachmanov and Zhou in \cite{RSZ}.

In other words, point ii) in this lemma guarantees that the error committed when replacing $y_j$ with $m_j$ is essentially the smallest possible. 

The second issue with this partition strategy is that the shape of the cells remains essentially rectangular only far from the poles. In particular the criterium for the choice of $y_j$ (and therefore of $m_j$), which is based on the fact that we seek for essentially square cells, deteriorates near the poles. 

The core idea of the present paper is to enhance the techniques of \cite{RSZ} so that taking the point distribution given by the centers of the cells we can get an average logarithmic energy as close as possible to \eqref{best}. In particular, we will use the method of \cite{RSZ} only for the region corresponding to latitudes between a fixed value $\Phi_0$ and $\pi-\Phi_0$ (we will call this part of the sphere tropical region). On the other hand, we will consider the polar regions, 
those with latitude smaller than $\Phi_0$ or larger than $\pi-\Phi_0$, as flat. Simple geometric observations show that one obtains equal area cells with very similar dimensions when the number of cells in each annulus in a flat configuration grows linearly. In this case, the diameter of each region is smallest when the number of cells in the $j$-th annulus is $6j$. Notice that in the Diamond ensemble the collars closest to the poles are divided precisely into $6j$ rectangular cells (that is $\beta_1=6$). In our case, we will divide the $j$-th polar collar into $8j$ rectangular cells. This still gives a rather small diameter of each polar cell and furthermore has the convenient property that the total number of cells in the first $j$ collars, including the cell around the North Pole, is 
\[1+8+16+\ldots+8j=(1+2j)^2.\]

In a forthcoming paper, we will show the geometric properties of this equal area partition, which in particular improves the best known results on the bounds of the diameter of the cells and of the diameter of the largest spherical cap contained in each cell (see \cite{GG}).

We will call Pearl ensemble the point process obtained from the uniformly drawn random rotation of each parallel in the above construction (see Definitions \ref{omega} and \ref{perla}). The origin of this name stems from the many analogies with the Diamond ensemble, but at the same time with a somehow slightly higher degree of smoothness in the approximation of $v_j$ vs. $y_j$ with integer values.

It should be mentioned here that, by \cite[Proposition 2.5]{BE}, for any given number of points in each parallel, the average logarithmic energy is minimized for certain specific computable values of the longitude. For this reason, rather than choosing precisely the centers of the cells, the nodes of the Pearl ensemble will be taken at the longitude given by \cite[Proposition 2.5]{BE}.

Our main result here is the following.

\begin{theorem}\label{main}
For every $\varepsilon>0$, there exists a choice of rational numbers $A_0=(\alpha/\sigma)^2$ and $B_0=(\beta/\tau)^2$, with $\alpha, \beta, \sigma, \tau$ integers, such that the corresponding Pearl ensemble with $N=\gamma(2h+1)^2$ points ($\gamma=\alpha^2\beta^2$ and $h$ is the integer parameter going to $+\infty$) has logarithmic energy with expected value %is the product  of the numerators of $A_0$ and $B_0$,
\[
\mathbb E( E_{N})=
\left(\frac12-\log2 \right)N^2-\frac12 N\log N+C(A_0,B_0)N+O(\sqrt{N}\log N),
\]
with
\[\frac{1-\log 3}{2}< C(A_0,B_0)<\frac{1-\log 3}{2}+\varepsilon.\]
\end{theorem}
Among the explicit constructions currently known to us for which the logarithmic energy expansion has been rigorously determined up to the linear term, the Pearl ensemble yields the smallest coefficient of the linear term.

\section{The Pearl ensemble}

For our purposes, Lemma \ref{lemmanondopato} is not sufficient. We need to show the following refined version.

\begin{lemma}\label{lemmadopato}
For every $n$ let $\{y_i\}_{i=1}^{2n+1}$ be a sequence of real numbers such that
\[y_j=y_{2n+2-j}, \qquad j=1, \ldots, n\] and such that \[\sum_{i=1}^{2n+1} y_i=T\in\mathbb N.\] 
Then there exists a symmetric sequence of integers $\{m_i\}_{i=1}^{2n+1}$ such that
\begin{itemize}
\item[$i)$] $\sum\limits_{i=1}^{2n+1} m_i=\sum  \limits_{i=1}^{2n+1} y_i=T$;
\item[$ii)$] $|y_i-m_i|\leq 2, \ i=1, \ldots, 2n+1;$
\item[$iii)$] $\left|\sum\limits_{i=1}^j (y_i-m_i) \right|\leq 1, \ j=1, \ldots, 2n+1$;
\item[$iv)$] $\left|\sum\limits_{k=1}^j \sum\limits_{i=1}^k (y_i-m_i) \right|\leq \dfrac12, \ j=1, \ldots, 2n+1$.
\end{itemize}
\end{lemma}

\begin{proof}
For every $j=1,\ldots n$ define
\[T_j=\sum\limits_{i=1}^j \sum\limits_{\ell=1}^i y_\ell.\]
Notice that
\[y_j=(T_j-T_{j-1})-(T_{j-1}-T_{j-2}).\]
Now, for every $j=1,\ldots n$, let $P_j$ the nearest integer to $T_j$, so that
\[|P_j-T_j|\leq \frac12,\]
and set $P_0=P_{-1}=0$. 
Define for every $j=1,\ldots n$
\[m_j=m_{2n+2-j}=(P_j-P_{j-1})-(P_{j-1}-P_{j-2})=P_j-2P_{j-1}+P_{j-2}\]
and
\begin{equation}\label{m}
m_{n+1}=T-2\sum_{j=1}^{n}m_j.
\end{equation}
Then, for every $j=1,\ldots n$,
\begin{align*}
\left|\sum_{i=1}^j\sum_{\ell=1}^{i}(y_\ell-m_\ell)\right|&=\left|\sum_{i=1}^j\sum_{\ell=1}^{i}y_\ell- \sum_{i=1}^j\sum_{\ell=1}^{i}m_\ell\right|=\left|T_j- \sum_{i=1}^j(P_i-P_{i-1})\right|=|T_j-P_j|\leq \frac12.
\end{align*}
Furthermore, for every $j=1,\ldots n$,
\begin{align*}
\left|\sum_{i=1}^j (y_i-m_i) \right|=\left|(T_j-T_{j-1})-(P_j-P_{j-1}) \right|\leq \frac12+\frac12=1
\end{align*}
and
\begin{align*}
|y_j-m_j|=|(P_j-P_{j-1})-(P_{j-1}-P_{j-2})-(T_j-T_{j-1})+(T_{j-1}-T_{j-2})|\leq 2.
\end{align*}
Now, by \eqref{m},
\[y_{n+1}-m_{n+1}=-2\sum_{j=1}^n (y_j-m_j).\]
Hence,
\[\sum_{j=1}^{n+1} (y_j-m_j)=\sum_{j=1}^n (y_j-m_j)+(m_{n+1}-y_{n+1})=-\sum_{j=1}^n (y_j-m_j)\]
and
\[\sum_{j=1}^{n+1}\sum_{i=1}^{j} (y_i-m_i)=\sum_{j=1}^{n-1}\sum_{i=1}^{j} (y_i-m_i)+\sum_{i=1}^{n} (y_i-m_i)+\sum_{i=1}^{n+1} (y_i-m_i)=\sum_{j=1}^{n-1}\sum_{i=1}^{j} (y_i-m_i).\]
Therefore
\[\left|\sum_{j=1}^{n+1}\sum_{i=1}^{j} (y_i-m_i)\right|\leq \frac12, \qquad \left|\sum_{j=1}^{n+1} (y_j-m_j)\right|\leq 1, \qquad |y_{n+1}-m_{n+1}|\leq 2.\]
By symmetry, we have the result for every $j> n+1$.
\end{proof}

Let $A, B\in\mathbb Q$, with $A>2$ and $B>0$. Let $N\in\mathbb N$ be the total number of points on the sphere. Let us assume that there exist $n_P$ and $n_T$ in $\mathbb N$ such that 
\begin{equation}\label{N}
N=A(2n_P+1)^2=B(2n_T+1)^2.
\end{equation}
Set 
\begin{equation}\label{Phizero}\Phi_0:=\arccos\left(1-\frac{2}A\right).\end{equation}
We shall call the regions of colatitude smaller than $\Phi_0$ or larger than $\pi-\Phi_0$ polar regions. We will consider $n_P$ parallels on each polar region and $2n_T+1$ parallels on the remaining tropical region. Concerning the polar regions, we will take $8j$ equally spaced points on the $j$-th parallel, $1\leq j\leq n_P$. As we mentioned in the introduction, this corresponds to a zonal partition of, say, the North polar region into rectangular cells, precisely $8j$ in the collar between latitudes $\varphi_{j-1}=\arccos(1-2(2j-1)^2/N)$ and $\varphi_{j}=\arccos(1-2(2j+1)^2/N)$, $1\leq j\leq n_P$. Concerning the tropical regions, let
\begin{equation}\label{Delta}
\Delta \psi=\frac{\pi-2\Phi_0}{2n_T+1}, \qquad \psi_j=\Phi_0+j\Delta \psi \ \text{ with } 0\leq j\leq 2n_T+{1} 
\end{equation}
and for $1\leq j\leq 2n_T+1$
\begin{equation}\label{y_j}y_j=\frac{N}{2}(\cos\psi_{j-1}-\cos\psi_j)=N\sin \theta_j\sin\left(\frac12\Delta\psi \right).
\end{equation}
with
\begin{equation}\label{theta_j}
\theta_j=\frac{\psi_{j}+\psi_{j-1}}{2}=\Phi_0+\left(j-\frac12\right)\Delta\psi.\end{equation}
Furthermore, for $1\leq j\leq 2n_T+1$,
\begin{equation}\label{s_j'}
s_j':=\sum_{i=1}^{j} y_i=\frac N2(\cos\psi_0-\cos\psi_j)=\frac N2-\frac NA-\frac N2\cos\psi_j.
\end{equation}
In particular,
\[s_{2n_T+1}'=\sum_{i=1}^{2n_T+1} y_i=N\left(1-\frac 2A\right)=N-2(2n_P+1)^2.\]

Let $\{m_j\}_{j=1}^{2n_T+1}$ be obtained as in Lemma \ref{lemmadopato} starting from the $\{y_j\}_{j=1}^{2n_T+1}$ just defined. Let 
\[1>z_{1}^P>z_{2}^P>\ldots>z_{n_P}^P>z_1^T>z_2^T>\ldots>z_{2n_T+1}^T>-z_{n_P}^P>\ldots>-z_1^P>-1.\]
Let us finally consider the generic sequences of angles $\{\vartheta_j^{NP}\}_{j=1}^{n_P}$,  $\{\vartheta_j^{SP}\}_{j=1}^{n_P}$,  and $\{\vartheta_j^T\}_{j=1}^{2n_T+1}$.

\begin{definition}\label{omega}
Let \[\Omega=\Omega(A,n_P, \{z_j^P\}_{j=1}^{n_P}, \{\vartheta_j^{NP}\}_{j=1}^{n_P}, \{\vartheta_j^{SP}\}_{j=1}^{n_P}, B,n_T,\{m_j\}_{j=1}^{2n_T+1},\{z_j^T\}_{j=1}^{2n_T+1}, \{\vartheta_j^T\}_{j=1}^{2n_T+1})\] be the following sets of $N$ points:
\begin{itemize}
\item[i.] the North and the South Poles, 
\item[ii.] for every $j=1,\ldots,n_P$ we consider $8j$ points equally spaced on a parallel at height $z_j^P$ starting from longitude $\vartheta_j^{NP}$ and $8j$ on the opposite parallel at height $-z_j^P$ starting from longitude $\vartheta_j^{SP}$, 
\item[iii.] for every $j=1,\ldots, 2n_T+1$ we consider $m_j$ points equally spaced on a parallel at height $z_j^T$ starting from longitude $\vartheta_j^{T}$.
\end{itemize}
\end{definition}

Observe that the total number of points in $\Omega$ is indeed $N$: two poles, $\sum_{j=1}^{n_P} 8j=4n_P(n_P+1)$ points in each polar region and  $\sum_{j=1}^{2n_T+1} m_j=\sum_{j=1}^{2n_T+1} y_j=N\left(1-\frac2A\right)$ points in the tropical region, that is
\[2(2n_P+1)^2+N\left(1-\frac2A\right)=2\frac NA+N\left(1-\frac2A\right)=N.\]
In particular the number of points in each polar region is $(2n_P+1)^2$, including the poles themselves. We will call this number $K$, so that
\[\sum_{j=1}^{2n_T+1} m_j=\sum_{j=1}^{2n_T+1} y_j=N-2K.\]

By \cite[Proposition 2.5]{BE} it is known that the  values of $z_j^P$ and $z_j^T$ that minimize the expected value for the logarithmic energy drawn from randomly and uniformly chosen rotations of each parallel are
\[z_j^P=1-\frac{1-8j+2\sum_{i=1}^{j} 8i}{N-1}=1-\dfrac{1+8j^2}{N-1} \quad \text{ for } j=1,\ldots,n_P,\]
\[z_j^T=1-\frac{1-m_j+2\left(\sum_{i=1}^{n_p}8i+\sum_{i=1}^jm_i\right)}{N-1}=1-\frac{2K-1-m_j+2s_j}{N-1} \quad \text{ for } j=1,\ldots,2n_T+1,\]
with $s_j=\sum_{i=1}^jm_i$. Notice that the following lemma holds.

\begin{lemma}
Each polar height $z_j^P$ lies in the corresponding interval:
\[\cos\varphi_j\leq z_j^P\leq \cos \varphi_{j-1}, \qquad 1\leq j\leq n_P.\]
Similarly, for $N$ sufficiently large, each tropical height $z_j^T$ lies in the corresponding interval:
\[\cos\psi_j\leq z_j^T\leq \cos \psi_{j-1}, \qquad 1\leq j\leq 2n_T+1.\]
\end{lemma}

\begin{proof}
Let us begin with the polar regions. It suffices to prove that 
\[1-\frac{2(2j+1)^2}N\leq1-\frac{1+8j^2}{N-1} \leq 1-\frac{2(2j-1)^2}N,\]
that is
\[\frac{1-8j}{2(2j+1)^2}\leq \frac{1}{N} \leq \frac{8j+1}{2(2j+1)^2}.\]
The first inequality is immediate, the second follows from the fact the the function $2(2j+1)^2/(8j+1)$ is increasing in $j$ and 
\[N\geq \frac{2(2n_P+1)^2}{8n_P+1}=\frac{2N}{A(8n_P+1)}\iff 8n_P+1\geq \frac{2}{A},\]
and this follows for each $n_p$ since $A>2$. 

Let us now consider the tropical regions and define
\begin{align*}
(z_j^T)'&:=1-\frac{2K-1-y_j+2s_j'}{N-1}\\
&=1-\frac{2\dfrac NA-1-\dfrac N2(\cos\psi_{j-1}-\cos\psi_j)+N-2\dfrac NA-N\cos\psi_j}{N-1}\\
&=\frac{N}{2(N-1)}(\cos\psi_{j-1}+\cos\psi_j),
\end{align*}
by \eqref{y_j} and \eqref{s_j'}.
Hence, by Lemma \ref{lemmadopato},
\begin{align*}
|z_j^T-(z_j^T)'|=\left|\frac{m_j-2s_j}{N-1}-\frac{y_j-2s_j'}{N-1} \right|=\left|\frac{-s_j-s_{j-1}+s_j'+s_{j-1}'}{N-1} \right|\leq \frac2{N-1},
\end{align*}
and therefore
\[z_j^T=\frac{N}{2(N-1)}(\cos\psi_{j-1}+\cos\psi_j)-\frac{(s_{j-1}-s'_{j-1})+(s_j-s'_j)}{N-1}.\]
Now, by \eqref{y_j},
\begin{align*}
z_j^T-\cos\psi_j&=\frac{N}{2(N-1)}(\cos\psi_{j-1}+\cos\psi_j)-\frac{(s_{j-1}-s'_{j-1})+(s_j-s'_j)}{N-1}-\cos\psi_j\\
&=\frac1{N-1}\left(\frac{N}{2}(\cos\psi_{j-1}+\cos\psi_j)-(N-1)\cos\psi_j-\left(s_{j-1}-s'_{j-1})-(s_j-s'_j)\right) \right)\\
&=\frac1{N-1}(y_j +\cos\psi_j-(s_{j-1}-s'_{j-1})-(s_j-s'_j))
\end{align*}
and similarly
\begin{align*}
\cos\psi_{j-1}-z_j^T=\frac1{N-1}(y_j -\cos\psi_{j-1}+(s_{j-1}-s'_{j-1})+(s_j-s'_{j})).
\end{align*}
Therefore, since $|(s_{j-1}-s'_{j-1})+(s_{j}-s'_{j})|\leq 2$, one has $\cos\psi_j\leq z_j^T\leq \cos\psi_{j-1}$ if and only if
\[y_j\geq \max\{2-\cos\psi_{j},\cos\psi_{j-1}+2\}.\]
The thesis follows noticing that 
\[y_j\geq N\sin\left(\Phi_0-\frac{\Delta\psi}{2}\right) \sin\left(\frac{\Delta\psi}{2}\right)\geq 3\]
for $N$ large enough since 
\[\Delta\psi=\frac{\sqrt{B}(\pi-2\Phi_0)}{\sqrt{N}}.\]
\end{proof}

We can now define the Pearl ensemble.

\begin{definition}\label{perla}
Let $A, B\in\mathbb Q$, with $A>2$ and $B>0$. Let $n_P, n_T \in \mathbb N$ be such that 
\[A(2n_P+1)^2=B(2n_T+1)^2.\]
Let $\{y_j\}$ and $\{m_j\}$ be defined as above. We define the \textbf{Pearl ensemble} as the family of random sets of points obtained from 
\[
\Omega(A,n_P, \{z_j^P\}_{j=1}^{n_P}, \{\vartheta_j^{NP}\}_{j=1}^{n_P}, \{\vartheta_j^{SP}\}_{j=1}^{n_P}, B,n_T,\{m_j\}_{j=1}^{2n_T+1},\{z_j^T\}_{j=1}^{2n_T+1}, \{\vartheta_j^T\}_{j=1}^{2n_T+1})
\] 
where $z_j^P$ and $z_j^T$ are as above, and each angle $\{\vartheta_j^{NP}\}_{j=1}^{n_P}, \{\vartheta_j^{SP}\}_{j=1}^{n_P}, \{\vartheta_j^T\}_{j=1}^{2n_T+1} $ is randomly uniformly drawn. More precisely, the points are the following
\[\begin{cases}
\mathcal N=(0,0,1)\\[5mm]
\left(\sqrt{1-(z_j^P)^2}\cos\left(\vartheta_j^{NP}+\dfrac{2\pi \ell}{8j} \right),\sqrt{1-(z_j^P)^2}\sin\left(\vartheta_j^{NP}+\dfrac{2\pi \ell}{8j} \right),z_j^P \right) & \begin{aligned}
\ell &= 0,\ldots,8j-1,\\
j &= 1,\ldots,n_P;
\end{aligned}
\\[5mm]
\left(\sqrt{1-(z_j^T)^2}\cos\left(\vartheta_j^{T}+\dfrac{2\pi \ell}{m_j} \right),\sqrt{1-(z_j^T)^2}\sin\left(\vartheta_j^{T}+\dfrac{2\pi \ell}{m_j} \right),z_j^T \right) & \begin{aligned}
\ell &= 0,\ldots,m_j-1,\\
j &= 1,\ldots,2n_T+1;
\end{aligned}
\\[5mm]
\left(\sqrt{1-(z_j^P)^2}\cos\left(\vartheta_j^{SP}+\dfrac{2\pi \ell}{8j} \right),\sqrt{1-(z_j^P)^2}\sin\left(\vartheta_j^{SP}+\dfrac{2\pi \ell}{8j} \right),-z_j^P \right) & \begin{aligned}
\ell &= 0,\ldots,8j-1,\\
j &= 1,\ldots,n_P;
\end{aligned}
\\[5mm]
\mathcal S=(0,0,-1).
\end{cases}\]
\end{definition}

\section{Proof of Theorem \ref{main}}

We begin with an explicit computation of the expected value of the logarithmic energy of the Pearl ensemble. 

\begin{theorem}\label{main2}
Let $A, B\in \mathbb Q$, with $A>2$ and $B>0$. Assume that
\[N=A(2n_P+1)^2=B(2n_T+1)^2.\]
Then the expected value for the logarithmic energy of the Pearl ensemble is
\[
\mathbb E(E_{N})=
\left(\frac12-\log2 \right)N^2-\frac12 N\log N+C(A,B)N+O(\sqrt{N}\log N),
\]
where
\begin{align*}
C(A,B)&=\frac{1}{3}\log\!\left(
\frac{8A^{7}}
{B^{\frac32}(\pi-2\Phi_0)^{3}(A-1)^{7}}
\right)
+\frac{1}{A}\log\!\left(
\frac{B(\pi-2\Phi_0)^{2}(A-1)}
{16A}
\right)
\\&-\frac{1}{A}
+\frac{B(A-2)(\pi-2\Phi_0)^{2}}
{24A}.
\end{align*}
\end{theorem}

By \cite[Theorem 2.6]{BE}, the expected value of the logarithmic energy is given by
\begin{align}\label{exenergy}
\mathbb E(E_{N})=&-(N-1)\log 4+m_{n_T+1}\log(m_{n_T+1})-2\left(\sum_{j=1}^{n_P} 8j\log(8j) + \sum_{j=1}^{n_T+1}m_j\log(m_j)\right)\nonumber\\
&-(N-1)\sum_{\pm}\sum_{j=1}^{n_P} 8j(1\pm z_j^P)\log(1\pm z_j^P)-(N-1)\sum_{j=1}^{2n_T+1} m_j(1- z_j^T)\log(1- z_j^T).
\end{align}
The proof of Theorem \ref{main2} is reduced to estimating five contributions:
\[
T_{eq}=m_{n_T+1}\log(m_{n_T+1})\]
\[T_{pol,1}=-2\sum_{j=1}^{n_P} 8j\log(8j), \qquad 
T_{pol,2}=-(N-1)\sum_{\pm}\sum_{j=1}^{n_P} 8j(1\pm z_j^P)\log(1\pm z_j^P),\]
\[T_{trop,1}=-2\sum_{j=1}^{n_T+1}m_j\log(m_j), \qquad
T_{trop,2}=-(N-1)\sum_{j=1}^{2n_T+1} m_j(1- z_j^T)\log(1- z_j^T).\]
We treat them separately in the following lemmas. Let us start with $T_{eq}$ and with the polar contributions.

\begin{lemma}\label{passo1} It holds
\[T_{eq}=O(\sqrt{N}\log N) \quad \text{ and } \quad T_{pol,1}=-8n_P^2\log(8n_P)+4n_P^2+O(\sqrt{N}\log N).\]
\end{lemma}

\begin{proof}
Notice that, by Lemma \ref{lemmadopato}, \eqref{Delta} and \eqref{y_j}, for every $j=1,\ldots,2n_T+1$,
\begin{equation}\label{my}
|m_j|\leq |y_j|+2\leq \frac{N}2 \frac{\pi}{2n_T+1}+2.
\end{equation}
By \eqref{N} this is $O(\sqrt{N})$ and the first thesis follows. 

Now set
$f(x)=8x\log(8x)$. Then
\begin{align*}
\left\vert \sum_{j=1}^{n_P} f(j)-\int_{0}^{n_P} f(x)dx\right\vert &=\left\vert\sum_{j=1}^{n_P} \int_{j-1}^j (f(j)-f(x))dx\right\vert\\
&\leq \sum_{j=2}^{n_P} \sup_{j-1\leq x\leq j}\left|f'(x)\right|+\int_{0}^1(8\log 8-8x\log(8x))dx\\
&= \sum_{j=2}^{n_P} 8(\log(8j)+1)+4\log8+2\\
&\leq 8(n_P-1)(\log(8n_P)+1)+4\log 8+2=O(\sqrt{N}\log N),
\end{align*}
by \eqref{N}. Finally,
\[\int_{0}^{n_P} f(x)dx=4n_P^2\log(8n_P)-2n_P^2.
\]
\end{proof}

We want now to estimate 
\begin{align*}T_{pol,2}&=-(N-1)\sum_{\pm}\sum_{j=1}^{n_P} 8j(1\pm z_j^P)\log(1\pm z_j^P)\\
&=-(N-1)\sum_{\pm}\sum_{j=1}^{n_P}8j\left(1\pm \left(1-\dfrac{1+8j^2}{N-1} \right) \right)\log\left(1\pm \left(1-\dfrac{1+8j^2}{N-1} \right) \right)\\
&=-(N-1)\sum_{\pm}\sum_{j=1}^{n_P}g_{\pm}(j),
\end{align*}
where
\[g_{\pm}(x)=8x\left(1\pm \left(1-\dfrac{1+8x^2}{N-1} \right) \right)\log\left(1\pm \left(1-\dfrac{1+8x^2}{N-1} \right) \right).\]
Following \cite{BE}, we estimate $T_{pol,2}$ by means of the trapezoidal rule.

\begin{lemma}\label{passo2} 
One has
\[\sum_{j=1}^{n_P} g_{\pm}(j) =\frac{g_{\pm}(n_P)}2+\int_{0}^{n_P} g_{\pm}(x)dx +\frac{g'_{\pm}(n_P)-g'_{\pm}(0)}{12}+O\left(N^{-1/2}\log N\right).
\]
\end{lemma}

\begin{proof}
Notice that
\[g_{+}(x)=8x\left(2-\dfrac{1+8x^2}{N-1}  \right)\log\left(2-\dfrac{1+8x^2}{N-1}  \right),\]
\[g_{-}(x)=8x\dfrac{1+8x^2}{N-1} \log\left(\dfrac{1+8x^2}{N-1} \right).\]
By the trapezoidal rule (see \cite[Lemma A.2]{BE}),
\[
\left|\sum_{j=1}^{n_P} g_{\pm}(j) +\frac{g_{\pm}(0)}2-\frac{g_{\pm}(n_P)}2-\int_{0}^{n_P} g_{\pm}(x)dx -\frac{g'_{\pm}(n_P)-g'_{\pm}(0)}{12}\right|\leq  n_P\frac{\zeta(3)}{4\pi^3}\sup_{0\leq x\leq n_P}\left|g^{(3)}_{\pm}(x)\right|.
\]
Since $n_P=O(\sqrt{N})$, $g_{\pm}(0)=0$ and
\[\sup_{0\leq x\leq n_P}\left|g^{(3)}_{-}(x)\right|=O\left(\frac{\log N}N\right) \quad \text{ and } \quad \sup_{0\leq x\leq n_P}\left|g^{(3)}_{+}(x)\right|=O\left(\frac{1}N\right),\]
then the thesis follows.
\end{proof}

\begin{proposition}\label{polar}
The total polar contribution satisfies
\[T_{pol,1}+T_{pol,2}=P_2(A)N^2+P_{\log}(A)N\log N+P_1(A)N+O(\sqrt{N}\log N),\]
where
\begin{align*}
P_2(A)&=\left(1-\frac{2}{A}\right)\log\left(1-\frac{1}{A} \right)+\frac{\log(A-1)}{A^2}+\frac{1-2\log2}{A},\\
P_{\log}(A)&=-\frac1A,\\
P_1(A)&=\left(\frac73-\frac1A \right)\log\left(\frac A{A-1} \right)+2\left(1-\frac1A \right)\log2-\frac1A.
\end{align*}
\end{proposition}

\begin{proof}
By Lemma \ref{passo1}, 
\[T_{pol,1}=-8n_P^2\log(8n_P)+4n_P^2+O(\sqrt{N}\log N).\]
Now we consider $T_{pol,2}$. By Lemma \ref{passo2},
\[(N-1)\sum_{j=1}^{n_P} g_{\pm}(j) =(N-1)\left(\frac{g_{\pm}(n_P)}2+\int_{0}^{n_P} g_{\pm}(x)dx +\frac{g'_{\pm}(n_P)-g'_{\pm}(0)}{12}\right)+O\left(\sqrt{N}\log N\right).
\]
One has
\[(N-1)g_{-}(n_P)=8n_P(1+8n_P^2) \log\left(\dfrac{1+8n_P^2}{N-1} \right)\]
and
\[(N-1)g_{+}(n_P)=8(N-1)n_P\left(2-\dfrac{1+8n_P^2}{N-1}  \right)\log\left(2-\dfrac{1+8n_P^2}{N-1}  \right).\]
Therefore
\begin{align*}
&\frac{N-1}2(g_{+}(n_P)+g_{-}(n_P))\\
&=8(N-1)n_P\log\left(2-\dfrac{1+8n_P^2}{N-1}  \right)-4n_P(1+8n_P^2) \left(\log\left(2-\dfrac{1+8n_P^2}{N-1}  \right)- \log\left(\dfrac{1+8n_P^2}{N-1} \right)\right).
\end{align*}
\begin{comment}
and hence
\begin{aligned}
&\frac{N-1}2\bigl(g_-(n_P)+g_+(n_P)\bigr)\\
&=
\underbrace{32n_P^3\log\left(\frac{8n_p^2}{N} \right)
+8Nn_P\log\!\left(2-\frac{8n_P^2}{N}\right)
-32n_P^3\log\!\left(2-\frac{8n_P^2}{N}\right)
}_{\mathcal{O}(N^{3/2})}
+o(N).
\end{aligned}
\end{comment}
Now, let us calculate
\[
(N-1)\int_0^{n_P} g_{-}(x)\,dx.
\]
Changing variable, 
$u=(1+8x^2)/(N-1)$, 
one has
\[
\int_0^{n_P} g_{-}(x)\,dx
=\frac{N-1}2\int_{1/(N-1)}^{(1+8n_P^2)/(N-1)}{u}\log\!\left({u}\right)\,du.
\]
Since
\[
\int u\log u\,du=\frac{u^2}{2}\log u-\frac{u^2}{4}+c,
\]
one obtains
\begin{align*}
(N-1)\int_0^{n_P} g_{-}(x)\,dx
&=
\frac14(1+8n_P^2)^2\log\!\left(\frac{1+8n_P^2}{N-1}\right)
-\frac18\Big((1+8n_P^2)^2-1\Big)
+\frac14\log(N-1)\\
&=
\frac14(1+8n_P^2)^2\log\!\left(\frac{1+8n_P^2}{N-1}\right)
-\frac18\Big((1+8n_P^2)^2-1\Big)+O(\log N).%\\
%&=16n_P^4\log\left(\frac{8n_P^2}{N} \right)-8n_P^4+4n_P^2\log\left(\frac{8n_P^2}{N} \right) +o(N).
\end{align*}
Similarly, we want to compute
\[
(N-1)\int_0^{n_P} g_{+}(x)\,dx.
\]
With the substitution
\[
u=2-\frac{1+8x^2}{N-1}, \qquad du=-\frac{16x}{N-1}dx,
\]
one obtains
\[
\int_0^{n_P} g_{+}(x)\,dx
=\frac{N-1}{2}\int_{2-\frac{1+8n_P^2}{N-1}}^{2-\frac1{N-1}} u\log u\,du,
\]
so that
\begin{align*}
&(N-1)\int_0^{n_P} g_{+}(x)\,dx
\\&=
\frac{(N-1)^2}{4}\Bigg[
\left(2-\frac1{N-1}\right)^2\log\left(2-\frac1{N-1}\right)
-\left(2-\frac{1+8n_P^2}{N-1}\right)^2\log\left(2-\frac{1+8n_P^2}{N-1}\right)
\Bigg]\\
&\quad -\frac{(N-1)^2}{8}\Bigg[
\left(2-\frac1{N-1}\right)^2
-\left(2-\frac{1+8n_P^2}{N-1}\right)^2
\Bigg]\\
&=(N-1)^2\left(\log 2-\frac12 \right)-(N-1)\log 2\\
&\quad -\frac{(N-1)^2}{4}\left(2-\frac{1+8n_P^2}{N-1}\right)^2\left(\log\left(2-\frac{1+8n_P^2}{N-1}\right)-\frac12\right)+O(1).%\\
%&=N^2\log 2-\frac{N^2}4\left(2-\frac{8n_P^2}N\right)^2\log\left(2-\frac{8n_P^2}N\right)-4Nn_P^2+8n_P^4+o(N).
\end{align*}
Therefore,
\[
\begin{aligned}
&(N-1)\left(\int_0^{n_P} g_{+}(x)\,dx+\int_0^{n_P} g_{-}(x)\,dx\right)\\
&=\frac14(1+8n_P^2)^2\log\!\left(\frac{1+8n_P^2}{N-1}\right)
-\frac18\Big((1+8n_P^2)^2-1\Big)+(N-1)^2\left(\log 2-\frac12 \right)-(N-1)\log 2\\
&\quad -\frac{(N-1)^2}{4}\left(2-\frac{1+8n_P^2}{N-1}\right)^2\left(\log\left(2-\frac{1+8n_P^2}{N-1}\right)-\frac12\right)+O(\log N)\\
%&=
%\underbrace{\left(
%16n_P^4\log\left(\frac{8n_P^2}{N}\right)
%-\frac{N^2}{4}\left(2-\frac{8n_P^2}{N}\right)^2
%\log\left(2-\frac{8n_P^2}{N}\right)
%+N^2\log{2}
%-4Nn_P^2
%\right)}_{\mathcal O(N^2)}
%\\[2mm]
%&\quad
%+\underbrace{\left(
%4n_P^2\log\left(\frac{8n_P^2}{N}\right)
%\right)}_{\mathcal O(N)}
%+o(N).
\end{aligned}
\]
Now, notice that
\begin{align*}
g'_{-}(0)&=-\frac{8}{N-1}\log(N-1),\\ g'_{-}(n_P)&=\frac{8(1+24n_P^2)}{N-1}\log\left(\frac{1+8n_P^2}{N-1}\right)+\frac{128n_P^2}{N-1}
\end{align*}
and
\begin{align*}
g'_{+}(0)&=8\left(2-\frac1{N-1} \right)\log\left(2-\frac1{N-1} \right)\\  g'_{+}(n_P)&=8\left(2-\frac{1+24n_P^2}{N-1} \right)\log\left(2-\frac{1+8n_P^2}{N-1} \right)-\frac{128n_P^2}{N-1}.
\end{align*}
It follows that
\begin{align*}
&\frac{N-1}{12}\left(g'_{+}(n_P)-g'_{+}(0)+g'_{-}(n_P)-g'_{-}(0) \right)\\&=
\frac{2}{3}(N-1)\left(
\frac{1+24n_P^2}{N-1}\log\left(\frac{1+8n_P^2}{N-1}\right)
+
\left(2-\frac{1+24n_P^2}{N-1}\right)
\log\left(2-\frac{1+8n_P^2}{N-1}\right)\right.
\\
&
\quad \left.-\left(2-\frac{1}{N-1}\right)\log\left(2-\frac{1}{N-1}\right)
+\frac{1}{N-1}\log(N-1)
\right)\\
&=
\frac{2}{3}
(1+24n_P^2)\log\left(\frac{1+8n_P^2}{N-1}\right)
+
\frac{2}{3}\left(2(N-1)-1-24n_P^2\right)
\log\left(2-\frac{1+8n_P^2}{N-1}\right)
\\
&
\quad -\frac{2}{3}\left(2(N-1)-1\right)\log\left(2-\frac{1}{N-1}\right)+O(\log N).
 %&=16n_P^2\log\left(\frac{8n_P^2}%{N}\right)+\left(\frac{4N}3-16n_P^2\right)
%\log\left(2-\frac{8n_P^2}{N}\right)-\frac{4N}{3}\log2-%\frac{32}3n_P^2+o(N).
\end{align*}
Summing everything together, 
\begin{align*}
&T_{pol,2}\\
&=-\frac14(1+8n_P^2)^2\log\!\left(\frac{1+8n_P^2}{N-1}\right)
+\frac18\Big((1+8n_P^2)^2-1\Big)-(N-1)^2\left(\log 2-\frac12 \right)+(N-1)\log 2\\
&\quad +\frac{(N-1)^2}{4}\left(2-\frac{1+8n_P^2}{N-1}\right)^2\left(\log\left(2-\frac{1+8n_P^2}{N-1}\right)-\frac12\right)\\
&\quad -8(N-1)n_P\log\left(2-\dfrac{1+8n_P^2}{N-1}  \right)+4n_P(1+8n_P^2) \left(\log\left(2-\dfrac{1+8n_P^2}{N-1}  \right)- \log\left(\dfrac{1+8n_P^2}{N-1} \right)\right)\\
&\quad -
\frac{2}{3}
(1+24n_P^2)\log\left(\frac{1+8n_P^2}{N-1}\right)
-
\frac{2}{3}\left(2(N-1)-1-24n_P^2\right)
\log\left(2-\frac{1+8n_P^2}{N-1}\right)
\\
&
\quad +\frac{2}{3}\left(2(N-1)-1\right)\log\left(2-\frac{1}{N-1}\right)
 +O(\sqrt{N}\log N).
\end{align*}
By \eqref{N}, considering $T_{pol,1}$ and $T_{pol,2}$, the thesis follows.
\end{proof}

Let us now consider the tropical contributions, starting from 
\begin{align*}
T_{trop,1}&=-2\sum_{j=1}^{n_T+1}m_j\log(m_j).
\end{align*}

\begin{lemma}\label{passo4} One has
\begin{align*}
T_{trop,1}&=-\frac{A-2}{2A}N\log N\\
&+\left[-\frac{A-2}A\log\left(\frac{(\pi-2\Phi_0)\sqrt{B}}{2} \right)+\frac{A-2}{A}\left(1+\log\frac A2 \right)-\frac{A-1}A\log(A-1)\right]N\\
&+O(\sqrt{N}\log N).
\end{align*}
\end{lemma}

\begin{proof} First, we want to replace $m_j$ with $y_j$. One has
\begin{align*}
&\left|\sum_{j=1}^{n_T+1}\left(m_j\log(m_j)-y_j\log(y_j) \right)\right|=\left|\sum_{j=1}^{n_T+1}\left((m_j-y_j)\log(m_j)+y_j\log\left(\frac{m_j}{y_j}\right) \right)\right|\\
&\leq \sum_{j=1}^{n_T+1}\left|(m_j-y_j)\log(m_j)\right|+\sum_{j=1}^{n_T+1}\left|y_j\log\left(\frac{m_j}{y_j}\right) \right| \lesssim n_T\log(n_T)+\sum_{j=1}^{n_T+1}\left|y_j\frac{m_j-y_j}{y_j} \right| \\
&=O(\sqrt{N}\log N),
\end{align*}
by Lemma \ref{lemmadopato} and since
\[\log\left(\frac{m_j}{y_j} \right)=\log\left(1+\frac{m_j-y_j}{y_j} \right).\]
Next, we replace $y_j$ with its linearized form
\[\widetilde{y_j}=\frac N2\Delta\psi\sin\theta_j .\]
Notice that
\begin{align*}
\left|y_j- \widetilde{y_j}\right|&= \left|N\sin\theta_j\left(\sin\left(\frac{\Delta\psi}2 \right)-\frac{\Delta\psi}2 \right) 
\right|\\
&\leq \left|N\sin\theta_j\frac{(\Delta\psi)^3}{48} 
\right|\leq NO\left(\frac1{n_T^3} \right)=O(N^{-\frac12}).
\end{align*}
Thus, the same argument as before shows that
\[\left|\sum_{j=1}^{n_T+1}y_j\log(y_j)- \sum_{j=1}^{n_T+1}\widetilde{y_j}\log(\widetilde{y_j})\right|=O(\sqrt{N}\log N).\]
Now, 
\begin{align*}
\sum_{j=1}^{n_T+1}\widetilde{y_j}\log(\widetilde{y_j})
&=\frac N2\log\left(\frac N2\Delta\psi \right)\sum_{j=1}^{n_T+1}  \Delta\psi\sin\theta_j+\frac N2\sum_{j=1}^{n_T+1}\Delta\psi\sin\theta_j\log\left(\sin\theta_j \right).
\end{align*}
These sums are the Riemann sums of the integrals 
\[\frac N2\log\left(\frac N2\Delta\psi \right)\int_{\Phi_0}^{\Phi_0+(n_T+1)\Delta\psi}\sin\theta d\theta + \frac N2\int_{\Phi_0}^{\Phi_0+(n_T+1)\Delta\psi}\sin\theta\log(\sin \theta)d\theta,\]
and, since $\sin
\theta$ and $\sin\theta\log(\sin\theta)$ have bounded derivatives, the error committed  replacing each sum with the corresponding integrals are $O(1/n_T)$, so that the overall error is 
\[O\left(N(\log N) \frac1{n_T}+N\frac1{n_T}\right)=O(\sqrt{N}\log N).\]
Since 
\[\Phi_0+(n_T+1)\Delta\psi=\frac \pi2+\frac{\Delta\psi}{2}=\frac\pi2+O\left(\frac1{n_T} \right),\]
it suffices to calculate
\begin{align*}
&\frac N2\log\left(\frac N2\Delta\psi \right)\int_{\Phi_0}^{\frac\pi2}\sin\theta d\theta + \frac N2\int_{\Phi_0}^{\frac\pi2}\sin\theta\log(\sin \theta) d\theta\\
&=\frac N2\log\left(\frac N2\Delta\psi \right) \cos\Phi_0-\frac N2\cos\Phi_0(1-\log(\sin \Phi_0))-\frac N2\log\left|\tan \frac{\Phi_0}2 \right|.
\end{align*}
By \eqref{N} and by the definition of $\Delta\psi$, we have the thesis.
\end{proof}

Let us now consider 
\begin{align*}
T_{trop,2}&=-(N-1)\sum_{j=1}^{2n_T+1} m_j(1- z_j^T)\log(1- z_j^T)\\
&= -(N-1)\sum_{j=1}^{2n_T+1} m_j\frac{2K-1-m_j+2s_j}{N-1}\log\left( \frac{2K-1-m_j+2s_j}{N-1}\right)\\
&=-(2K-1)\sum_{j=1}^{2n_T+1} m_jL_j+\sum_{j=1}^{2n_T+1} m_j^2L_j -2\sum_{j=1}^{2n_T+1} m_js_jL_j,
\end{align*}
where
\[L_j=\log\left(\frac{2K-1-m_j+2s_j}{N-1}\right).\]
Notice that $L_j=O(1)$ since
\begin{equation}\label{Lj}1-\cos\Phi_0\leq \frac{2K-1-m_j+2s_j}{N-1}=1-z_j^T< 2.
\end{equation}

The idea now is to replace in the above sums the values $m_j$ (and therefore $s_j$) with the more explicit values $y_j$ (and $s_j'$).

\begin{proposition}\label{passo5}
Let
\[L_j'=\log\left(\frac{2K-1-y_j+2s'_j}{N-1}\right).\]
Then
\begin{align*}
\sum_{j=1}^{2n_T+1} m_jL_j&=\sum_{j=1}^{2n_T+1} y_jL_j'+O(\sqrt{N})\\
\sum_{j=1}^{2n_T+1} m_j^2L_j&=\sum_{j=1}^{2n_T+1} y_j^2L_j'+O(\sqrt{N})\\
\sum_{j=1}^{2n_T+1} m_js_jL_j&=\sum_{j=1}^{2n_T+1} y_js_j'L_j'+O(\sqrt{N}).
\end{align*}
\end{proposition}

\begin{proof}
First, let us consider the ``linear'' term. One has
\begin{align*}
&\sum_{j=1}^{2n_T+1} \left|m_jL_j-y_jL_j'\right|=\sum_{j=1}^{2n_T+1} \left|(m_j-y_j)L_j+y_j\log\left(\frac{2K-1-m_j+2s_j}{2K-1-y_j+2s'_j} \right)\right|= O(\sqrt{N}).
\end{align*}
This follows from Lemma \ref{lemmadopato}, from \eqref{Lj} and from
\begin{equation}\label{log}\log\left(\frac{2K-1-m_j+2s_j}{2K-1-y_j+2s'_j} \right)=\log\left(1+\frac{y_j-m_j-2(s_j'-s_j)}{2K-1-y_j+2s'_j} \right)=O\left(\frac1N\right).
\end{equation}

Second, we consider the ``quadratic'' term. Notice that
\begin{align*}
&\left|\sum_{j=1}^{2n_T+1} m_j^2L_j-y_j^2L_j'\right|=\left|\sum_{j=1}^{2n_T+1} (m_j^2-y_j^2)L_j+y_j^2\log\left(\frac{2K-1-m_j+2s_j}{2K-1-y_j+2s'_j} \right)\right|.
\end{align*}
Letting for simplicity $e_j=m_j-y_j$, we can write
\[
m_j^2-y_j^2=(m_j-y_j)(m_j+y_j)
=e_j(2y_j+e_j),
\]
so that we have
\[
\sum_{j=1}^{2n_T+1}
(m_j^2-y_j^2)
L_j=2\sum_{j=1}^{2n_T+1}e_jy_jL_j
+
\sum_{j=1}^{2n_T+1}e_j^2L_j.
\]
Therefore by \eqref{Lj} and Lemma \ref{lemmadopato},
\[
\sum_{j=1}^{2n_T+1}e_j^2L_j=O(n_T)=O(\sqrt{N}).
\]
Let us now consider
\[
\sum_{j=1}^{2n_T+1}e_jy_jL_j.
\]
By summation by parts,
\begin{align*}
\sum_{j=1}^{2n_T+1}e_jy_jL_j&=y_{2n_T+1}L_{2n_T+1}\sum_{i=1}^{2n_T+1}e_i-\sum_{j=1}^{2n_T}
\left(y_{j+1}L_{j+1}-y_jL_j\right)\sum_{i=1}^j
e_i\\&=-\sum_{j=1}^{2n_T}
\left(y_{j+1}L_{j+1}-y_jL_j\right)\sum_{i=1}^j
e_i.
\end{align*}
Now,
\[
y_{j+1}L_{j+1}-y_jL_j
=
(y_{j+1}-y_j)L_j
+
y_{j+1}(L_{j+1}-L_j).
\]
Notice that
\[
L_{j+1}-L_j=\log\left(\frac{2K-1-m_{j+1}+2s_{j+1}}{2K-1-m_j+2s_j} \right)=\log\left(1+\frac{m_j+m_{j+1}}{2K-1-m_j+2s_j} \right)=O(N^{-1/2})
\]
and
\begin{align*}
|y_{j+1}-y_j|&=\left|N\sin\left(\frac{\Delta \psi}{2}\right)\left[\sin\left(\Phi_0+\left(j+\frac12\right)\Delta\psi \right)-\sin\left(\Phi_0+\left(j-\frac12\right)\Delta\psi \right)\right]\right|\\
&=\left|2N\sin^2\left(\frac{\Delta \psi}{2}\right)\cos(\Phi_0+j\Delta\psi)\right|\leq 2N\sin^2\left(\frac{\Delta \psi}{2}\right)=O\left(N(\Delta\psi)^2 \right)=O(1).
\end{align*}
Hence,
\begin{equation}\label{diff}
y_{j+1}L_{j+1}-y_jL_j
=
(y_{j+1}-y_j)L_j
+
y_{j+1}(L_{j+1}-L_j)
=
O(1)
\end{equation}
and, by Lemma \ref{lemmadopato},
\[
\sum_{j=1}^{2n_T+1}e_jy_jL_j
=
O(\sqrt N).
\]
In conclusion,
\[
\sum_{j=1}^{2n_T+1}
(m_j^2-y_j^2)
L_j
=
O(\sqrt N)
\]
and by \eqref{log}
\[\left|\sum_{j=1}^{2n_T+1} (m_j^2-y_j^2)L_j+y_j^2\log\left(\frac{2K-1-m_j+2s_j}{2K-1-y_j+2s'_j} \right)\right|=O(\sqrt{N}).\]

Finally, we have to consider the ``mixed'' term. As before, one has
\begin{align*}
&\left|\sum_{j=1}^{2n_T+1} m_js_jL_j-y_js'_jL_j' \right|=\left|\sum_{j=1}^{2n_T+1} (m_js_j-y_js'_j)L_j+y_js'_j\log\left(\frac{2K-1-m_j+2s_j}{2K-1-y_j+2s'_j} \right)\right|.
\end{align*}
First, consider 
\[\sum_{j=1}^{2n_T+1} (m_js_j-y_js'_j)L_j.\]
For simplicity, let us set $S_j=\sum_{i=1}^j e_i$ for $j=1,\ldots, 2n_T+1$, and $S_0=0$. We have
\[
m_js_j-y_js'_j
=
(y_j+e_j)(s'_j+S_j)-y_js'_j
=
e_js'_j+y_jS_j+e_jS_j.
\]
Therefore
\[
\sum_{j=1}^{2n_T+1}(m_js_j-y_js'_j)L_j
=
\sum_{j=1}^{2n_T+1}e_js'_jL_j
+
\sum_{j=1}^{2n_T+1}y_jS_jL_j
+
\sum_{j=1}^{2n_T+1}e_jS_jL_j.
\]
Remember that
\[
e_j=O(1),\qquad S_j=O(1),
\qquad 
\sup_k\left|\sum_{j=1}^k S_j\right|=O(1),
\]
\[
y_j=O(\sqrt N),\qquad y_{j+1}-y_j=O(1),
\qquad s'_j=O(N),
\]
\[
L_j=O(1),\qquad L_{j+1}-L_j=O(N^{-1/2}).
\]
The estimate for the third sum follows easily, 
\[
\sum_{j=1}^{2n_T+1}e_jS_jL_j=O(n_T)=O(\sqrt N).
\]
Let us consider the second sum. Recall that by \eqref{diff}
\[y_{j+1}L_{j+1}-y_jL_j=O(1),\]
so that
\[
\sum_{j=1}^{2n_T}|y_{j+1}L_{j+1}-y_jL_j|=O(\sqrt N).
\]
By summation by parts,
\[\sum_{j=1}^{2n_T+1}S_jy_jL_j=y_{2n_T+1}L_{2n_T+1}\sum_{j=1}^{2n_T+1}S_j-\sum_{j=1}^{2n_T}(y_{j+1}L_{j+1}-y_jL_j)\sum_{i=1}^{j} S_i=O(\sqrt{N}).\]
Now, consider the first sum. Let us introduce the following notation. For any sequence $a_j$ with $j\in\mathbb Z$, we shall call $\Delta a_j=a_{j+1}-a_{j}$. If the sequence is finite, say $\{a_j\}_{j=1}^k$ we shall assume $a_j=0$ for $j<1$ and for $j>k$. Since $e_j=S_j-S_{j-1}$, and $S_0=S_{2n_T+1}=0$,
\[
\sum_{j=1}^{2n_T+1}e_js'_jL_j
=
\sum_{j=1}^{2n_T+1}(S_j-S_{j-1})s'_jL_j=-\sum_{j=1}^{2n_T}\Delta F_jS_j
\]
where
\[F_j=s'_jL_j.\]
Now,
\[\Delta F_j=\Delta s_j'L_{j+1}+ s_j'\Delta L_j\]
and
\[\Delta^2 F_j=\Delta^2 s_j'L_{j+2}+2\Delta s_j'\Delta L_{j+1}+s_j'\Delta^2 L_j\]
with
\begin{align*}
\Delta s_j'&=y_{j+1}=O(N^{1/2}),\\
\Delta^2 s_j'&=y_{j+2}-y_{j+1}=O(1),\\
\Delta L_{j+1}&=\log\left(\frac{2K-1-m_{j+2}+2s_{j+2}}{2K-1-m_{j+1}+2s_{j+1}} \right)=O(N^{-1/2}),\\
\Delta^2 L_{j}&=\log\left(1+\frac{m_{j+2}+m_{j+1}}{2K-1-m_{j+1}+2s_{j+1}} \right)+\log\left(1-\frac{m_{j+1}+m_{j}}{2K-1-m_{j+1}+2s_{j+1}}\right)\\
&=\frac{m_{j+2}+m_{j+1}}{2K-1-m_{j+1}+2s_{j+1}}-\frac{m_j+m_{j+1}}{2K-1-m_{j+1}+2s_{j+1}} +O(N^{-1})\\
&=\frac{m_{j+2}-m_{j}}{2K-1-m_{j+1}+2s_{j+1}}+O(N^{-1})=O(N^{-1}),
\end{align*}
since $m_{j+2}-m_j=O(1)$ because $y_{j+2}-y_j=O(1)$ and $|m_j-y_j|\leq 2$.
Therefore,
\[\Delta F_j=O(\sqrt{N}), \qquad \Delta^2 F_j=O(1). \]
Again, by summation by parts
\[\sum_{j=1}^{2n_T}\Delta F_jS_j=\Delta F_{2n_T}\sum_{j=1}^{2n_T}S_j-\sum_{j=1}^{2n_T-1}\Delta^2F_j\sum_{i=1}^j S_i=O(\sqrt{N})\]
In conclusion,
\[
\sum_{j=1}^{2n_T+1}
(m_js_j-y_js'_j)
L_j
=
O(\sqrt N).
\]
Let us consider now
\[\sum_{j=1}^{2n_T+1} y_js'_j\log\left(\frac{2K-1-m_j+2s_j}{2K-1-y_j+2s'_j} \right)=\sum_{j=1}^{2n_T+1} y_js'_j\log\left(1+\frac{y_j-m_j+2(s_j-s_j')}{2K-1-y_j+2s'_j} \right)\]
and its expansion
\[\sum_{j=1}^{2n_T+1} \frac{y_js'_j}{2K-1-y_j+2s'_j}\left(y_j-m_j+2(s_j-s_j')\right)+O(1).\]
Let us call
\[W_j:=\frac{y_js'_j}{2K-1-y_j+2s'_j}\]
and notice that $W_j=O(\sqrt{N})$. By summation by parts and by Lemma \ref{lemmadopato}, we have
\begin{align*}
&\sum_{j=1}^{2n_T+1} W_j\left(y_j-m_j+2(s_j-s_j')\right)\\&=W_{2n_T+1}\sum_{j=1}^{2n_T+1}\left(-e_j+2S_j\right)-\sum_{j=1}^{2n_T}(W_{j+1}-W_j)\sum_{i=1}^{j}\left(-e_i+2S_i\right)=O(\sqrt{N})
\end{align*}
since
\[W_{j+1}-W_j=O(1).\]
Indeed,
\begin{align*}
W_{j+1}-W_j&=\frac{y_{j+1}s'_{j+1}}{2K-1-y_{j+1}+2s'_{j+1}}-\frac{y_js'_j}{2K-1-y_j+2s'_j}\\
&=\frac{(y_{j+1}-y_j)s'_j}{2K-1-y_j+2s'_j}+\frac{y_{j+1}(s'_{j+1}-s'_j)}{2K-1-y_j+2s'_j}\\
&\quad +y_{j+1}s'_{j+1}\left(\frac1{2K-1-y_{j+1}+2s'_{j+1}}-\frac1{2K-1-y_j+2s'_j}\right)\\
&=\frac{(y_{j+1}-y_j)s'_j}{2K-1-y_j+2s'_j}+\frac{y_{j+1}(s'_{j+1}-s'_j)}{2K-1-y_j+2s'_j}\\
&\quad +y_{j+1}s'_{j+1}\left(\frac{y_{j+1}-y_j-2(s'_{j+1}-s'_{j})}{(2K-1-y_{j+1}+2s'_{j+1})(2K-1-y_j+2s'_j)}\right)\\
&=\frac{O(1)O(N)}{O(N)}+\frac{O(\sqrt{N})O(\sqrt{N})}{O(N)}+O(N^{\frac32})\frac{O(\sqrt{N})}{O(N^2)}.
\end{align*}
\end{proof}

By Proposition \ref{passo5}, we may replace $m_j$ with $y_j$ in the definition of $T_{trop,2}$, so that
\[T_{trop,2}=-(N-1)\sum_{j=1}^{2n_T+1} y_j(1-(z_j^T)')\log(1-(z_j^T)')+O(\sqrt{N})\]
where, by \eqref{y_j} and \eqref{s_j'},
\begin{align*}
(z_j^T)'&:=1-\dfrac{2K-1-y_j+2s'_j}{N-1}=\frac N{N-1}\frac{\cos\psi_{j-1}+\cos\psi_j}{2}\\&=\frac N{N-1}\cos\theta_j\cos\left(\frac{\Delta\psi}2 \right)=\frac N{N-1}\cos\theta_j\left(1-\frac{(\Delta\psi)^2}{8}\right) +O(N^{-2}).
\end{align*}
Notice that, by \eqref{y_j},
\[y_j=N\sin\theta_j\left(1-\frac{(\Delta \psi)^2}{24}\right)\frac{\Delta\psi}2+O(N^{-\frac32}).\]
Therefore
\begin{align*}
&y_j\left(1-(z_j^T)'\right)\log\left(1-(z_j^T)'\right)\\
&=\frac N2\Delta\psi\left(\sin\theta_j\left(1-\frac{(\Delta \psi)^2}{24}\right)+O(N^{-2}) \right)\left(1-\frac N{N-1}\cos\theta_j\left(1-\frac{(\Delta\psi)^2}{8}\right) +O(N^{-2}) \right)\\
&\quad \times\log\left(1-\frac N{N-1}\cos\theta_j\left(1-\frac{(\Delta\psi)^2}{8}\right) +O(N^{-2}) \right)=\frac N2\Delta\psi f_N(\theta_j)+O(N^{-\frac32}),
\end{align*}
where
\begin{align*}
f_{N}(\theta)&=\sin\theta\left(1-\frac{(\Delta \psi)^2}{24}\right)\left(1-\frac N{N-1}\cos\theta\left(1-\frac{(\Delta\psi)^2}{8}\right)  \right)\\
&\quad \times\log\left(1-\frac N{N-1}\cos\theta\left(1-\frac{(\Delta\psi)^2}{8}\right)  \right).
\end{align*}
Hence,
\[T_{trop,2}=-\frac{N(N-1)} 2\sum_{j=1}^{2n_T+1} f_N(\theta_j)\Delta\psi+O(\sqrt{N}).\]
Once again we will estimate this Riemann sum by means of the corresponding integral. 

\begin{lemma}\label{passo8} It holds
\begin{align*}
&T_{trop,2}=-\frac{N(N-1)}2\int_{\Phi_0}^{\pi-\Phi_0} f_N(\theta)d\theta+\frac{N(N-1)}{48}(\Delta\psi)^2(f'_N(\pi-\Phi_0)-f'_N(\Phi_0))+O(\sqrt{N}),
\end{align*}
\end{lemma}

\begin{proof}

Notice that $|f_N^{(\ell)}(\theta)|\leq C$ uniformly in $\theta$ and in $N$ for $\ell=0,\ldots, 4$. Then, an iterated application of the midpoint rule to $f_N$ and $f''_N$ gives
\begin{align*}
T_{trop,2}&=-\frac{N(N-1)}2\int_{\Phi_0}^{\pi-\Phi_0} f_N(\theta)d\theta+\frac{N(N-1)}2\frac{(\Delta\psi)^2}{24}(f'_N(\pi-\Phi_0)-f'_N(\Phi_0))\\&\quad+O(N^2(\Delta\psi)^4) +O(\sqrt{N})
\end{align*}
and
$O(N^2(\Delta\psi)^4) =O(1)$.
\end{proof}

\begin{proposition}\label{tropical}
The total tropical contribution satisfies
\[T_{trop,1}+T_{trop,2}=T_2(A)N^2+T_{\log}(A)N\log N+T_1(A,B)N+O(\sqrt{N}\log N),\]
where
\begin{align*}
T_2(A)&=-\frac{(A-1)^2}{A^2}\log\left(\frac{2(A-1)}{A} \right)+\frac1{A^2}\log\left(\frac2A \right)+\frac{A-2}{2A},\\
T_{\log}(A)&=-\frac{A-2}{2A},\\
T_1(A,B)&=\frac{A-2}{A}\left(\log\left(\frac{2}{(\pi-2\Phi_0)\sqrt{B}}\right)+\frac{B(\pi-2\Phi_0)^2}{24} \right).
\end{align*}
\end{proposition}

\begin{proof}
The expansion for $T_{trop,1}$ is in Lemma \ref{passo4}.

Let us consider the expansion of $T_{trop,2}$ in Lemma \ref{passo8}. For the first term, by the change of variable
\[u=1-\frac N{N-1}\cos\theta\left(1-\frac{(\Delta\psi)^2}{8}\right),\]
one has
\begin{align*}
\frac{N(N-1)}2\int_{\Phi_0}^{\pi-\Phi_0} f_N(\theta)d\theta&=\frac{(N-1)^2}2 \frac{1-\dfrac{(\Delta\psi)^2}{24}}{1-\dfrac{(\Delta\psi)^2}{8}}\\
&\times\left(\frac{u_+^2}{2}\log u_+-\frac{u^2_-}{2}\log u_- -\frac N{N-1}\left(1-\frac{(\Delta\psi)^2}8 \right)\cos\Phi_0\right),
\end{align*}
with
\[u_{\pm}=1\pm \frac{N}{N-1}\left(1-\frac{(\Delta\psi)^2}8 \right)\cos\Phi_0.\]
Expanding in decreasing powers of $N$ and recalling \eqref{N}, we have
\begin{align*}
&-\frac{N(N-1)}2\int_{\Phi_0}^{\pi-\Phi_0} f_N(\theta)d\theta\\
&=-\frac{N^2}{2}
\left[
\frac{(1+\cos\Phi_0)^2}{2}\log(1+\cos\Phi_0)
-\frac{(1-\cos\Phi_0)^2}{2}\log(1-\cos\Phi_0)
-\cos\Phi_0
\right]\\
&
\quad -N\Bigg\{
\left(
-1+\frac{B(\pi-2\Phi_0)^2}{24}
\right)
\\
&\quad \times\left[
\frac{(1+\cos\Phi_0)^2}{2}\log(1+\cos\Phi_0)
-\frac{(1-\cos\Phi_0)^2}{2}\log(1-\cos\Phi_0)
-\cos\Phi_0
\right]\\
&
\quad +\frac{\cos\Phi_0}{2}
\left(
1-\frac{B(\pi-2\Phi_0)^2}{8}
\right)
\Big[
(1+\cos\Phi_0)\log(1+\cos\Phi_0)
+(1-\cos\Phi_0)\log(1-\cos\Phi_0)
\Big]
\Bigg\}\\
&
\quad +O(1).
\end{align*}
Recalling that 
\[\cos\Phi_0=1-\frac2A,\] one obtains
\begin{align*}
&-\frac{N(N-1)}2\int_{\Phi_0}^{\pi-\Phi_0} f_N(\theta)d\theta\\
&=-N^2\left(\frac{(A-1)^2}{A^2}\log\left(\frac{2(A-1)}{A} \right)-\frac1{A^2}\log\left(\frac2A \right)-\frac{A-2}{2A}\right)\\
&\quad -N\left\{2\left(-1+\frac{B(\pi-2\Phi_0)^2}{24} \right)\left(\frac{(A-1)^2}{A^2}\log\left(\frac{2(A-1)}{A} \right)-\frac1{A^2}\log\left(\frac2A \right)-\frac{A-2}{2A}\right)\right\}\\
&\quad \quad\left.+\frac{A-2}{A^2}\left(1-\frac{B(\pi-2\Phi_0)^2}{8} \right)\left[(A-1)\log\left(\frac{2(A-1)}A\right)+\log\left(\frac2A \right) \right]\right\}+O(1).
\end{align*}
For the second term, since
\begin{align*}
f_N'(\theta)
&=
\left(1-\frac{(\Delta\psi)^2}{24}\right)
\Bigg[
\cos\theta
\left(
1-\frac N{N-1}\cos\theta
\left(1-\frac{(\Delta\psi)^2}{8}\right)
\right) \\
&\quad \times
\log\left(
1-\frac N{N-1}\cos\theta
\left(1-\frac{(\Delta\psi)^2}{8}\right)
\right) \\
&\quad +
\frac N{N-1}
\left(1-\frac{(\Delta\psi)^2}{8}\right)
\sin^2\theta
\left(
\log\left(
1-\frac N{N-1}\cos\theta
\left(1-\frac{(\Delta\psi)^2}{8}\right)
\right)+1
\right)
\Bigg],
\end{align*}
we have
\begin{align*}
&\frac{N(N-1)}2\frac{(\Delta\psi)^2}{24}
\left(f'_N(\pi-\Phi_0)-f'_N(\Phi_0)\right)\\
&=
\frac{N(N-1)}{48}(\Delta\psi)^2
\left(1-\frac{(\Delta\psi)^2}{24}\right)\\
&\quad \times
\Bigg[
-\cos\Phi_0
\Bigg(
\left(1+\frac N{N-1}\cos\Phi_0
\left(1-\frac{(\Delta\psi)^2}{8}\right)\right)
\log\left(
1+\frac N{N-1}\cos\Phi_0
\left(1-\frac{(\Delta\psi)^2}{8}\right)
\right)
\\
&\quad +
\left(1-\frac N{N-1}\cos\Phi_0
\left(1-\frac{(\Delta\psi)^2}{8}\right)\right)
\log\left(
1-\frac N{N-1}\cos\Phi_0
\left(1-\frac{(\Delta\psi)^2}{8}\right)
\right)
\Bigg)
\\
&\quad
+\frac N{N-1}
\left(1-\frac{(\Delta\psi)^2}{8}\right)
\sin^2\Phi_0
\log\left(
\frac{
1+\frac N{N-1}\cos\Phi_0
\left(1-\frac{(\Delta\psi)^2}{8}\right)}
{
1-\frac N{N-1}\cos\Phi_0
\left(1-\frac{(\Delta\psi)^2}{8}\right)}
\right)
\Bigg].
\end{align*}
Therefore, expanding in decreasing powers of $N$,
\begin{align*}
&\frac{N(N-1)}{2}\frac{(\Delta\psi)^2}{24}
\left(f_N'(\pi-\Phi_0)-f_N'(\Phi_0)\right)\\
&=
B\frac{(\pi-2\Phi_0)^2}{48}
\Bigg[
-\cos\Phi_0
\Big(
(1+\cos\Phi_0)\log(1+\cos\Phi_0)
+(1-\cos\Phi_0)\log(1-\cos\Phi_0)
\Big)
\\
&\quad +\sin^2\Phi_0
\log\left(
\frac{1+\cos\Phi_0}{1-\cos\Phi_0}
\right)
\Bigg]N
+O(1)\\
&=\frac{B(\pi-2\Phi_0)^2}{24A^2}\left[(A^2-4A+2)\log\left(\frac A2 \right)+(A-1)(4-A)\log(A-1) \right]N+O(1).
\end{align*}
Combining the expansion of $T_{trop,1}$ with the two contributions above and simplifying, the result follows.
\end{proof}

We can now prove Theorem \ref{main2}.

\proof[Proof of Theorem \ref{main2}.] 
The decomposition of the expected energy \eqref{exenergy}, by Proposition \ref{polar} and Proposition \ref{tropical}
\begin{align*}
\mathbb E(E_N)&=-(N-1)\log 4+T_{eq}+ T_{pol,1}+T_{pol,2}+T_{trop,1}+T_{trop,2}\\
&=(P_2(A)+T_2(A))N^2+(P_{\log}(A)+T_{\log}(A))N\log N\\
&\quad +(-2\log 2+P_1(A)+T_1(A,B))N+O(\sqrt{N}\log N).
\end{align*}
By Lemma \ref{passo1}, Proposition \ref{polar} and Proposition \ref{tropical}, since
\begin{align*}
P_2(A)+T_2(A)&=\frac12-\log2, \quad \text{ and } \quad P_{\log}(A)+T_{\log}(A)=-\frac12,
\end{align*}
we obtain
\[
\mathbb E(E_{N})=
\left(\frac12-\log2 \right)N^2-\frac12 N\log N+C(A,B)N+O(\sqrt{N}\log N),
\]
where
\begin{align*}
C(A,B)&=-2\log2+P_1(A)+T_1(A,B)\\
&=\frac{1}{3}\log\!\left(
\frac{8A^{7}}
{B^{\frac32}(\pi-2\Phi_0)^{3}(A-1)^{7}}
\right)
+\frac{1}{A}\log\!\left(
\frac{B(\pi-2\Phi_0)^{2}(A-1)}
{16A}
\right)
\\&\quad -\frac{1}{A}
+\frac{B(A-2)(\pi-2\Phi_0)^{2}}
{24A}.
\end{align*}
\endproof

\begin{comment}
\begin{remark} Let us suppose that
there is no tropical part (except at the equator, and that term is $o(N)$). In this case 
Hence
\[\mathbb E(\mathcal E_{\log})
= \left(\frac12-\log 2 \right)N^2-\frac12 N\log N+N\left(\log 2-\frac12+\frac43\log 2-\frac32\log 2 \right)+o(N),
\]
that is exactly the result in \cite[section 4.1]{BE}.
\end{remark}
\end{comment}

\proof[Proof of Theorem \ref{main}]
For every fixed $A$, the value of $B$ which minimizes $C(A,B)$ is
\[B_{\min}(A)=\frac{12}{(\pi-2\Phi_0)^2}.\]
In this case,
\begin{align*}
C(A,B_{\min}(A))
&=
\left(\frac{7}{3}-\frac{1}{A}\right)
\log\left(\frac{A}{A-1}\right)
+\left(\frac{1}{A}-\frac{1}{2}\right)\log 3
-\frac{2\log 2}{A}
+\frac{1}{2}
-\frac{2}{A}.
\end{align*}
Notice that $C(A,B_{\min}(A))$ is decreasing for $A>2$ and if $A\to+\infty$, 
\[C(A,B_{\min}(A))\to \frac{1-\log3}{2}.\]
Therefore,
\[\inf_{A> 2, B>0} C(A,B)=\frac{1-\log3}{2}.\]
Thus, for every $\varepsilon>0$, there exist $A_0> 2$ and $B_0>0$ in $\{(a/b)^2:\ a,b \text{ odd integers}\}$ such that 
\[C(A_0,B_0)\leq \frac{1-\log3}{2}+\varepsilon.\]
Let
\[A_0=\left(\frac{\alpha}{\sigma}\right)^2 \qquad B_0=\left(\frac{\beta}{\tau}\right)^2.\]
Since we need
\[N=\left(\frac{\alpha}{\sigma}\right)^2(2n_P+1)^2=\left(\frac{\beta}{\tau}\right)^2(2n_T+1)^2,\]
that is
\[\alpha\tau(2n_P+1)=\beta\sigma(2n_T+1),\]
it suffices to take 
\[2n_T+1=(2h+1)\alpha\tau, \qquad 2n_P+1=(2h+1)\beta\sigma\]
with $h$ integer. In other words,
\[N=\alpha^2\beta^2(2h+1)^2.\]
\endproof

\begin{remark} 
The total number of parallels in the Pearl ensemble is 
\[2n_P+2n_T+1=\sqrt{N}\left(\frac{1}{\sqrt{A}}+\frac1{\sqrt{B}}\right)-1.\]
Notice that 
\[\lim_{A\to+\infty}\left(\frac{1}{\sqrt{A}}+\frac1{\sqrt{B_{\min}(A)}}\right)=\frac{\pi}{\sqrt{12}}\sim 0.906\ldots.\]
In the heuristics from \cite{RSZ} described in the Introduction, the total number of parallels is 
\[2n+1\sim \frac{\sqrt{\pi N}}2\]
with
\[\frac{\sqrt{\pi}}2\sim 0.886\ldots\]
Notice that
\[\frac{2n_P+2n_T+1}{2n+1}=\sqrt{\frac{\pi}{3}}\sim 1.023\ldots\]
This means that the optimal Pearl ensemble has the $2.3\%$ of parallels more than the construction in \cite{RSZ}. 
\end{remark}

%\sim 

\end{document}